\documentclass[a4,11pt]{article}
\usepackage{amsmath,amsthm,mathrsfs,amssymb}
\newcommand{\C}{\mathbb{C}}

\newcommand{\N}{\mathbb{N}}
\newcommand{\R}{\mathbb{R}}
\newcommand{\Z}{\mathbb{Z}}
\newcommand{\E}{\mathbb{E}}

\newcommand{\T}{\mathcal{T}}
\newcommand{\den}{\mathcal{D}}
\newcommand{\myrare}{\smash[b]{\underline{{\mathcal{R}}}}_r(\xi)}
\newcommand{\myden}{\smash[b]{\underline{{\mathcal{D}}}}_r(\xi)}

\newcommand{\set}[1]{\left\{#1\right\}}
\newcommand{\e}{\epsilon}
\renewcommand{\S}{\mathcal{S}}
\renewcommand{\P}{\mathbb{P}}
\renewcommand{\index}{i \mspace{-4mu} \imath}

\newtheorem{thm}{Theorem}
\newtheorem*{thm*}{Theorem}
\newtheorem{lem}{Lemma}
\newtheorem{prop}[thm]{Proposition}
\newtheorem*{prop*}{Proposition}
\newtheorem{cor}{Corollary}
\newtheorem*{spec}{Spectral control}

\theoremstyle{definition}
\newtheorem{rem}{Remark}
\newtheorem{Def}{Definition}

\makeatletter
\def\appendix{% 
\def\theequation{\Alph{section}.\arabic{equation}}% 
\setcounter{equation}{0}% 
\def\thesection{\Alph{section}}% 
\@addtoreset{equation}{section}% 
\setcounter{section}{0}% 
\def\@seccntformat##1{% 
\@nameuse{prefix@##1}% 
\@nameuse{the##1}% 
\@nameuse{postfix@##1}\quad}% 
\def\prefix@section{Appendix~}%  
} 
\makeatother

\begin{document}

\title{Brownian survival and Lifshitz tail in perturbed lattice disorder }
\author{Ryoki Fukushima}
\date{\today}
\maketitle

\begin{abstract}
We consider the annealed asymptotics for the survival probability of Brownian motion among randomly distributed traps. 
The configuration of the traps is given by independent displacements of the lattice points. 
We determine the long time asymptotics of the logarithm of the survival probability up to a multiplicative constant. 
As applications, we show the Lifshitz tail effect of the density of states of the associated random Schr\"{o}dinger operator 
and derive a quantitative estimate for the strength of intermittency in the parabolic Anderson problem.\\
\textbf{Keywords}: Brownian motion; random media; perturbed lattice; random Schr\"odinger operators; Lifshitz tail; 
random displacement model\\
\textbf{MSC 2000 subject classification}: 60K37; 60G17; 82D30; 82B44\\
\end{abstract} 

%%%%%%%%%%%%%%%%%%%%%%%%%%%%%%%%%%%%%%%%%%%%%%%%%%section1%%%%%%%%%%%%%%%%%%%%%%%%%%%%%%%%%%%%%%%%%%%
% main text
\section{Introduction}\label{intro}
We consider the annealed asymptotics for the survival probability of Brownian motion among randomly distributed traps. 
This problem for the Poissonian configuration of traps was firstly investigated by Donsker and Varadhan \cite{DV75c}
and later by Sznitman \cite{Szn90a} with generalizations on the shape of each trap, the diffusion coefficient, 
and the underlying space. 
Sznitman also generalized the configuration to some Gibbsian point processes in \cite{Szn93b}. 

In this article, we discuss another model where the traps are attached around a randomly perturbed lattice. 
Namely, our process is the killed Brownian motion whose generator is 
\begin{equation}
 H_{\xi}=-\frac{1}{2}\Delta + \sum_{q \in \Z^d} W(\,\cdot-q-\xi_q), \label{ptl}
\end{equation}
where $(\xi_q)_{q \in \Z^d}$ is a collection of i.i.d.\ random vectors and $W$ is a nonnegative function 
whose support is compact and has nonempty interior. 
We allow $W$ to take the value $\infty$, which means imposing the Dirichlet boundary condition on $\set{W=\infty}$. 
If $W\equiv\infty$ on its support, the traps are said to be {\it hard}. 
The random potential in \eqref{ptl} is a model of the ``Frenkel disorder'' in solid state physics and is called 
the ``random displacement model'' 
in the theory of random Schr\"odinger operator. For such models with \textit{bounded} displacements, 
there are some results concerning the spectral properties of the generator. 
Kirsch and Martinelli \cite{KM82b} discussed the existence of band gaps 
and Klopp \cite{Klo93} proved the spectral localization in a semi-classical limit. 
More recently, Baker, Loss and Stolz \cite{BLS07} studied 
which configuration minimizes the bottom of the spectrum of \eqref{ptl}. 
On the other hand, there are few results when the displacements are \textit{unbounded}, which is the object of 
this article. 
It seems important to allow unbounded displacements from the physical viewpoint. 
For example, it is natural to take the Gaussian distribution for the displacements 
to model the defects caused by self-diffusion. 
In the future paper \cite{FU08}, we will extend the investigation to 
non-compactly supported potentials and negative potentials. 
We will also discuss in \cite{FU08} the one-dimensional result which is not discussed in the present article. 

There are at least three important aspects of the survival probability. The first is as the partition function 
for the Brownian motion conditioned to survive. Actually, some detailed studies on the surviving Brownian motion 
were developed after \cite{Szn90a}. See e.g.~\cite{Szn91b} and \cite{Pov99} for path localization results. 
The second is as the Laplace transform of the density of states. It is well known that one can derive the asymptotic 
behavior of the density of states near the bottom of the spectrum from the survival asymptotics using an exponential 
Tauberian theorem. 
See e.g.\ \cite{Fuk73}, \cite{Nak77}, and \cite{Szn90a} for this way of studies on 
the density of states. 
The last is as the solution of the parabolic Anderson problem. The quenched survival probability 
of the Brownian motion is expressed by a Feynman-Kac functional. 
From the expression, we can identify it with the solution of the heat equation associated with $H_{\xi}$. 
Therefore, the annealed asymptotics of the survival probability gives the moment asymptotics of the solution. 

Now we describe the settings precisely. Let $((\xi_q)_{q \in \Z^d}, \P_{\theta})$ ($\theta > 0$) be 
$\R^d$-valued i.i.d.\ random variables with the distribution 
\begin{equation}
 \P_{\theta}(\xi_q \in dx) = N(d,\theta) \exp\{-|x|^{\theta} \}\,dx \label{tail}, 
\end{equation}
where $N(d,\theta)$ is the normalizing constant. 
Although our proof needs such an assumption only on the tail, we assume $\xi_q$ to have the 
exact density \eqref{tail} for simplicity. 
The parameter $\theta$ controls the strength of the disorder: 
large $\theta$ implies weak disorder and small $\theta$ implies the converse. 
Given random vectors, we define the perturbed lattice by 
$\xi = \sum_{q \in \Z^d} \delta_{q+\xi_q}$ and let $V(\,\cdot\, ,\xi)$ be the random potential in \eqref{ptl}. 
We denote by $\Xi$ the sample space of $\xi$, the space of simple pure point measures on $\R^d$. 
We use the notation $((B_t)_{t \ge 0}, P_x)$ for the standard Brownian motion which is independent of $\xi$. 
The entrance time to a closed set $F$ and the exit time from an open set $U$ are denoted by $H_F$ and $T_U$, respectively. 
Then the survival probability, our main object of this article, is described as follows: 
\begin{equation*}
 S_t = \E_{\theta} \otimes E_0 \left[ \exp\set{-\int_0^t V(B_s,\xi) \,ds}\right]. 
\end{equation*}
Intuitively, this quantity seems to decay exponentially since the traps are distributed 
almost uniformly in the space. However, the decay rate should be slower than the periodic case 
since large {\it trap free regions} caused by the disorder help the Brownian survival. 

We make a remark on the starting point of the Brownian motion before stating the results. 
Since our trap field is not $\R^d$-translation invariant but $\Z^d$-shift invariant, 
the asymptotics of the survival probability may depend on the starting point. 
However,  it will be clear from the proof that all the results stated in this article do not 
depend on the starting point. For this reason, we only consider the Brownian motion starting 
from the origin.  

%%%%%%%%%%%%%%%%%%%%%%%%%%%%%%%%%%%Results%%%%%%%%%%%%%%%%%%%%%%%%%%%%%%%%%%%%%%%%%%%%%%%%%%%%%%%%
We discuss the long time asymptotics of $\log S_t$, instead of $S_t$ itself, in this article. 
We introduce some notations for asymptotic behaviors to state the results. 
\begin{Def}
Let $f$ and $g$ are real-valued functions of a real variable 
and $\ast = 0$ or $\infty$. Then
\begin{equation*}
 f(x) \asymp g(x) \quad {\rm as} \quad x \to \ast 
\end{equation*}
means that there exists a constant $C>0$ such that
\begin{equation}
 C^{-1} g(x) \le f(x) \le C g(x) 
\end{equation}
when $x$ is sufficiently close to $\ast$. Similarly, 
\begin{equation}
 f(x) \sim g(x) \quad {\rm as} \quad x \to \ast 
\end{equation}
means that $\lim_{x \to \ast} f(x)/g(x) = 1$. 
\end{Def}
We are now ready to state our main result. 
\begin{thm}\label{thm1}
For any $\theta > 0$, we have 
 \begin{equation*}
  \begin{split}
   \log S_t \asymp 
   \left\{
   \begin{array}{lr}
    - t^{\frac{2+\theta}{4+\theta}} (\log t)^{-\frac{\theta}{4+\theta}} &\quad (d=2),\\[8pt]
    -t^{\frac{d^2+2\theta}{d^2+2d+2\theta}}  &\quad (d \ge 3),
   \end{array}\right.
  \end{split}
 \end{equation*}
as $t \to \infty$.  
\end{thm}

Our result says that the survival probability decays faster than 
in the Poissonian case where $\log S_t \sim -ct^{d/(d+2)}$ (cf.~\cite{DV75c}). This implicitly implies that the perturbed lattice is more 
\emph{ordered} than the Poisson point process.  
Furthermore, we have the following simple but interesting observations. 
\begin{rem}\label{disorder} (\textit{weak and strong disorder limits})\\
Concerning the power of $t$ in Theorem \ref{thm1}, we have the following: 
\begin{enumerate}
\item{As $\theta \nearrow \infty$, the power $\frac{d^2+2\theta}{d^2+2d+2\theta} \nearrow 1$, 
 which is the same as for the periodic traps.}
\item{As $\theta \searrow 0$, the power $\frac{d^2+2\theta}{d^2+2d+2\theta} \searrow \frac{d}{d+2}$, 
 which is the same as for the Poissonian traps.}
\end{enumerate}
On the other hand, a logarithmic correction remains in the case $d=2$ and $\theta \uparrow \infty$, 
which we do not have for the periodic traps. 
\qed \end{rem}
This remark says that our model can be regarded as an interpolation between a perfect crystal and 
a completely disordered medium. 
We will also show the similar results concerning the convergence of point processes 
in Appendix~\ref{A}. 
\begin{rem}
The perturbed lattice has another interesting aspect in the two-dimensional case. 
Let $Z_{\C}$ be the flat chaotic analytic zero points (CAZP), that is, the zero points of the Gaussian entire function 
$f_{\C}(z) = \sum_{n=0}^{\infty} a_n z^n/\sqrt{n!}$ where $(a_n)_{n = 0}^{\infty}$ is 
a collection of i.i.d.\ standard complex Gaussian variables. 
Sodin and Tsirelson \cite{ST06} proved that there exists a collection of random variables $(\zeta_q)_{q \in \Z^2}$ 
such that $\sum_{q \in \Z^2}\delta_{\sqrt{\pi}q + \zeta_q}$ has the same distribution as $Z_{\C}$. 
Though $(\zeta_q)_{q \in \Z^2}$ is not an independent family, it is invariant under lattice shifts and 
the distribution of each $\zeta_q$ has a Gaussian upper bound for the tail. 
Therefore, our model with the parameter $\theta=2$ can be regarded as a toy model 
for the flat CAZP. Indeed, Sodin and Tsirelson called our model ``the second toy model'' in \cite{ST06}. 
\qed \end{rem}

Let us briefly explain the construction of the article. 
We prove Theorem \ref{thm1} in Section \ref{proof}. 
Our strategy to prove the survival asymptotics is based on the idea in \cite{Szn90a, Szn98} 
rather than the one in \cite{DV75c}. 
The first step is a reduction to a certain variational problem. 
In this step, we use a coarse graining method which is a slightly altered version of Sznitman's 
``method of enlargement of obstacles''. 
The second step is the analysis of the variational problem. 
However, we reverse the order and analyze the variational problem first 
since it gives the correct scale which we need in the coarse graining. 
In Section \ref{applications}, we give two applications of the survival asymptotics. 
The first is the Lifshitz tail effect on the density of states of $H_{\xi}$, which says 
that the spectrum of $H_{\xi}$ is exponentially thin around the bottom (cf.~\cite{Lif65}). 
The second is a quantitative estimate for the strength of intermittency for the solution of 
the parabolic Anderson problem associated with $H_{\xi}$. 

%%%%%%%%%%%%%%%%%%%%%%%%%%%%%%%%%%%%%%%%%%%%%%%%%%section2%%%%%%%%%%%%%%%%%%%%%%%%%%%%%%%%%
\section{Proof of the survival asymptotics}\label{proof}
%%%%%%%%%%%%%%%%%%%%%%%%%%%%%%%%%%%%%%%%%%%%%%%%%%section2.1%%%%%%%%%%%%%%%%%%%%%%%%%%%%%%%
\subsection{Rough procedure}\label{outline}
We explain the rough procedure of the proof in this section. 
First of all, we slightly modify the random potential as follows: 
\begin{equation}
 V(x,\xi) = \sum_{q \in \Z^d} h \cdot 1_{\{\xi(C(\e q,\epsilon)) \ge 1\}} 1_{C(\e q,L)}(x)
 +\infty \cdot 1_{\T^c}(x), \label{mptl}
\end{equation}
where $C(y,l) = y+[-l/2,l/2]^d$ and $\T = (-t/2,t/2)^d$. 
This new potential bounds the original one from both above and below in $\T$ by taking small $\epsilon$ 
and varying $h \in (0,\infty]$ and $L>0$. 
Moreover, the restriction on $\T$ does not affect the results since $P_0(T_{\T} \le t)$ 
decays exponentially in $t$. 
Therefore it is sufficient to prove the survival asymptotics for the modified potential \eqref{mptl}. 
Hereafter we take $\epsilon, h, L = 1$ so that $V(x,\xi) = 1_{{\rm supp}V(\,\cdot, \, \xi)}(x)$ almost everywhere in $\T$, 
for simplicity. 
We start with the following obvious lower and upper bounds. \\[10pt]
\noindent
\textit{Lower bound}: Let $\S$ be the set of possible shapes of ${\rm supp}\,V(\,\cdot, \xi)_{\xi \in \Xi}$. 
Then, 
\begin{equation*} 
 S_t \ge \sup_{U \in \mathcal{S}} \P_{\theta}(\xi(U^c)=0) 
 E_0\left[\exp \set{-\int_0^t 1_U(B_s) \,ds};T_{\mathcal{T}} > t\right]. \label{lower}
\end{equation*}

\noindent
\textit{Upper bound}: By summing over $U \in \mathcal{S}$, we obtain
\begin{equation*}
 \begin{split}
  S_t &\le\, \sum_{U \in \mathcal{S}} \P_{\theta}(\xi(U^c)=0) E_0\left[\exp \set{-\int_0^t 1_U(B_s) \,ds}\right]\\
  &\le\, \# \mathcal{S}  \sup_{U \in \mathcal{S}} \P_{\theta}(\xi(U^c)=0)
  E_0 \left[\exp \set{-\int_0^t 1_U(B_s) \,ds}; T_{\mathcal{T}} > t\right]. \label{upper}
 \end{split}
\end{equation*}

Here we have $\# \mathcal{S} < \infty$ thanks to above modification and therefore the upper bound makes sense. 
However, there still remains a problem since we have too many configurations: $\# \mathcal{S} \sim 2^{t^d}$. 
We shall remedy this situation by reducing $\# \mathcal{S}$ to the small order using a coarse graining method. 
Once $\# \S$ is shown to be negligible, the proof of the survival asymptotics is reduced to the analysis of 
the variational problem
\begin{equation}
 \sup_{U \in \mathcal{S}} \P_{\theta}(\xi(U^c)=0)E_0 \left[\exp \set{-\int_0^t 1_U(B_s) \,ds}; T_{\mathcal{T}} > t\right].
 \label{vp1}
\end{equation} 
As we announced in the introduction, we shall analyze this variational problem in Section \ref{variational} and 
give the coarse graining scheme in Section \ref{MEO}. Finally, we shall patch them together in Section \ref{patching} 
to complete the proof. 
\begin{rem}
For $\log S_t$ with the above modified potential, 
we can derive a finer asymptotics than Theorem \ref{thm1}. 
We shall state this in Section \ref{patching} (Theorem  \ref{thm3}) since it requires the notation defined in the proof. 
\qed \end{rem}
%
%%%%%%%%%%%%%%%%%%%%%%%%%%%%%%%%%%%%%%%%%%%%%%%%%%section2.2%%%%%%%%%%%%%%%%%%%%%%%%%%%%%%%%%
\subsection{Analysis of the variational problem}\label{variational}
In this section, we analyze the variational problem \eqref{vp1} and find the correct scale. 
Firstly, it is well known that the Brownian expectation part is controlled by the principal eigenvalue $\lambda_1(U)$ 
of the Dirichlet-Schr\"{o}dinger operator $-1/2\Delta + 1_U$ in $\T$: 
\begin{equation*}
 \log E_0 \left[\exp \set{-\int_0^t 1_U(B_s) \,ds}; T_{\mathcal{T}} > t\right] \sim -\lambda_1(U)t 
 \quad \textrm{as}\quad t \to \infty, \label{spectral}
\end{equation*}
for fixed $U$. Let us assume for the moment that this relation holds uniformly 
in $U \in \S$. We will give a rigorous argument in Section \ref{patching}. 
On the other hand, we use the following lemma to control the \emph{emptiness probability}, 
that is, the probability of the perturbed lattice putting no point in a region. 
%%%%%%%%%%%%%%%%%%%%%%%%%%%%%%%%%%%%%%%%%%%%%%%%%%emptiness probability%%%%%%%%%%%%%%%%%%%%%%%%%%%%%%%%
\begin{lem}\label{lem1}
 If $\{U_v\}_{v > 0} \subset \mathcal{S}$ satisfies 
 $\int_{U_v^c} {\rm d}(q, \partial U_v)^{\theta} dx/|U_v^c| \to \infty$ as $v \to \infty$, 
 then we have 
 \begin{equation}
  \log \P_{\theta}(\xi(U_v^c)=0) \sim -\int_{U_v^c} {\rm d}(x, \partial U_v)^{\theta} dx 
  \quad {\rm as} \quad v \to \infty, \label{hole}
 \end{equation}
 where ${\rm d}(\cdot, \cdot)$ denotes the Euclidean distance. 
\end{lem}
%%%%%%%%%%%%%%%%%%%%%%%%%%%%%%%%%%%%%%%%%%%%%%%%%%%%%%%%%%%%%%%%%%%%%%%%%%%%%%%%%%%%%%%%%%%%%%%%%
\begin{proof}
 Let $\e \in (0,1)$ and $U \in \S$ be fixed. Note that $|U^c|<\infty$ since $U^c$ is contained in $\T$. 
 
 For the upper bound, we consider the probability of a necessary condition: 
 \begin{equation}
  \begin{split}
   &\,\P_{\theta}\left(|\xi_q| > {\rm d}(q, \partial U) \textrm{ for all } q \in U^c \cap \Z^d \right)\\
   =&\prod_{q \in {U^c} \cap \Z^d}\int_{|x|>{\rm d}(q, \partial U)}N(d,\theta) \exp\bigl\{-|x|^{\theta}\bigr\}\,dx\\
   =&\prod_{q \in {U^c} \cap \Z^d} \sigma_{d} \int_{{\rm d}(q, \partial U)}^{\infty}
    N(d,\theta) r^{d-1} \exp\bigl\{-r^{\theta}\bigr\}\,dr\\
   \le &\prod_{q \in {U^c} \cap \Z^d} M_1(\e) \int_{{\rm d}(q, \partial U)}^{\infty}
    (1-\e) \theta r^{\theta-1} \exp\bigl\{-(1-\e)r^{\theta}\bigr\}\,dr\\
   =&\, M_1(\e)^{\# U^c \cap \Z^d} \exp\biggl\{-(1-\e)\sum_{q \in U^c \cap \Z^d} {\rm d}(q, \partial U)^{\theta}\biggr\}.
   \label{Lem1-upper}
  \end{split}
 \end{equation}
 Here $\sigma_{d}$ is the surface area of the unit sphere in $\R^d$ and 
 \begin{equation*}
  M_1(\e)= \frac{N(d,\theta) \sigma_{d}}{(1-\e)\theta} \sup_{r>1/2} r^{d-\theta} \exp\bigl\{ -\e r^{\theta} \bigr\}<\infty. 
 \end{equation*}
 We can replace the sum in the last line of \eqref{Lem1-upper} by the integral by making $M_1(\e)$ larger since
 \begin{equation*} 
  \sup_{U \in \mathcal{S} ,\, q \in U^c \cap \Z^d} 
  \Bigl\{\int_{C(q,1)}{\rm d}(x, \partial U)^{\theta}dx - {\rm d}(q, \partial U)^{\theta}\Bigr\} < \infty. 
 \end{equation*}
 Therefore we obtain 
 \begin{equation*}
  \P_{\theta}(\xi(U^c)=0) \le M_1(\e)^{|U^c|}
  \exp\Bigl\{-(1-\e) \int_{U^c} {\rm d}(x, \partial U)^{\theta} \,dx \Bigr\} 
 \end{equation*}
 for arbitrary $\e > 0$ and the upper bound follows. 

 For the lower bound, we consider a sufficient condition: 
\begin{equation}
 \begin{split}
  \P_{\theta}(\xi(U^c)=0) 
  \ge &\, \prod_{q \in U^c \cap \Z^d} \P_{\theta}(q+\xi_q \in \textrm{a nearest }C(q',1) \not\subset U^c)\\
  & \quad \times \prod_{q \in \Z^d \cap U ; \, {\rm d}(q, \partial U) \le M_2} \P_{\theta}(q+\xi_q \in C(q,1))\\
  & \qquad \times \prod_{q \in \Z^d \cap U ; \, {\rm d}(q, \partial U) > M_2} \P_{\theta}(q+\xi_q \notin U^c). 
  \label{Lem1-lower}
 \end{split}
\end{equation}
The first factor of the right-hand side is bounded below by 
\begin{equation}
 \begin{split}
  &c_0(d,\theta)^{|U^c|} \exp\biggl\{- \sum_{q \in U^c \cap \Z^d} {\rm d}(q, \partial U)^{\theta}\biggr\}\\
  \ge &\, \exp\Bigl\{-\int_{U^c} {\rm d}(x, \partial U)^{\theta} dx - |U^c|\cdot|\log c_0(d,\theta)|\Bigr\}
  \label{Lem1-1st}
 \end{split}
\end{equation} 
for some constant $c_0(d,\theta)>0$. For instance, it suffices to take $c_0(d,\theta)$ as
\begin{equation*}
\begin{split}
 (N(d,\theta)\wedge 1) \inf\{&\exp\{{\rm d}(x_1, y_1)^{\theta}-{\rm d}(x_2, y_2)^{\theta}\};\\
 &q, q' \in \Z^d, x_1 ,x_2 \in C(q,1), y_1, y_2 \in C(q',1) \}.
\end{split}
\end{equation*}
Next, the second factor is bounded below by
\begin{equation}
 \P_{\theta}(q+\xi_q \in C(q,1))^{(3M_2)^d|U^c|},\label{Lem1-2nd}
\end{equation}
since we have $\# \{q \in \Z^d \cap U ; \, {\rm d}(q, \partial U) \le M_2 \} \le (3M_2)^d |U^c|$ 
by considering the $M_2$-neighborhood of each unit cube contained in $U^c$. 
Before proceeding the estimate for the third factor, we recall that we have shown in \eqref{Lem1-upper} that 
\begin{equation*}
 \P_{\theta}(q+\xi_q \in U^c) \le M_1(\e)\exp\bigl\{-(1-\e){\rm d}(q, \partial U)^{\theta}\bigr\}
\end{equation*}
for any $\e>0$ and $q \in U$. 
Now, if we pick some $\e_0 \in (0,1)$ (e.g.\ $\e_0=1/2$) and take $M_2$ so large that 
\begin{equation*}
 M_1(\e_0)\exp\set{-(1-\e_0) (M_2-1)^{\theta}} < 1, 
\end{equation*}
then the third factor is bounded below by
\begin{equation}
 \begin{split}
  &\prod_{n \ge M_2} \prod_{q \in \Z^d; \, n-1 \le {\rm d}(q,\partial U) <n} 
  \left( 1-M_1(\e_0) \exp \set{-(1-\e_0) (n-1)^{\theta}} \right)\\
  \ge & \prod_{n \ge M_2} \left( 1-M_1(\e_0) \exp \set{-(1-\e_0) (n-1)^{\theta}} \right)^{(2n+1)^d|U^c|}\\
  \ge & \bigg(\prod_{n \ge M_2} \left( 1-M_1(\e) (2n+1)^d \exp \set{-(1-\e_0) (n-1)^{\theta}} \right)\biggr)^{|U^c|}, 
  \label{Lem1-3rd}
 \end{split}
\end{equation}
where we have used $\# \{q \in \Z^d; \, n-1 \le {\rm d}(q,\partial U) <n \} \le (2n+1)^d|U^c|$ in the second line 
and the elementary inequality $(1-x)^m \ge 1-mx$ ($x \in [0, 1]$, $m \in \N$) in the last line. 
Note that the infinite product in the third line is convergent. 

Combining \eqref{Lem1-lower}--\eqref{Lem1-3rd}, we obtain 
 \begin{equation*}
  \P_{\theta}(\xi(U^c)=0) \ge 
  \exp\Bigl\{- \int_{U^c} {\rm d}(x, \partial U)^{\theta} \,dx - c_0'(d,\theta)|U^c| \Bigr\}, 
 \end{equation*}
which shows the lower bound. 
\end{proof}
\noindent
Lemma \ref{lem1} says that \eqref{hole} holds for a large class of families in $\S$. 
In fact, we shall prove in Proposition \ref{prop6} that the family of sets 
which are \emph{relevant} in our analysis satisfies the assumption of Lemma \ref{lem1}. 
If we assume that 
 \begin{equation}
  \log \P_{\theta}(\xi(U^c)=0) = -\int_{U^c} {\rm d}(x, \partial U)^{\theta} dx 
 \end{equation}
holds together with \eqref{spectral} for all $U \in \S$, we can rewrite our variational problem as
\begin{equation} 
 \begin{split}
  &\,\log \sup_{U \in \S} \P_{\theta}(\xi(U^c)=0)E_0 \left[\exp \set{-\int_0^t 1_U(B_s) \,ds}; T_{\mathcal{T}} > t\right]\\
  & \sim \, -\inf_{U \in \S}\set{ \lambda_1(U)t +  
  \int_{U^c} {\rm d}(x, \partial U)^{\theta}dx }.\label{vp2}
 \end{split}
\end{equation}
It is easy to see that the infimum of \eqref{vp2} is attained when $U^c$ is large for large $t$. 
Thus, it is convenient to introduce a scaling $U=rU_r$ by a factor $r>0$.  
Under this scaling, the right-hand side of \eqref{vp2} takes the form
\begin{equation}
\begin{split}
 &-\inf_{U_r \in \S_r}\set{\lambda_1^r(U_r)tr^{-2} + r^{d+\theta} 
 \int_{U_r^c} {\rm d}(x, \partial U_r)^{\theta} \,dx}\\
 &=-tr^{-2}\inf_{U_r \in \S_r}\set{\lambda_1^r(U_r) + \frac{r^{d+\theta}}{tr^{-2}}
 \int_{U_r^c} {\rm d}(x, \partial U_r)^{\theta} \,dx}.\label{vp3}
\end{split}
\end{equation}
Here $\S_r=\set{r^{-1}U ; U \in \S}$ and $\lambda_1^r(U_r)$ is the principal eigenvalue of 
the scaled Dirichlet-Schr\"{o}dinger operator $-1/2\Delta+r^2 1_{U_r}$ in $\T_r=r^{-1}\T$. 

Let us summarize the status. We have shown that 
\begin{equation}
 \log S_t \sim -tr^{-2}\inf_{U_r \in \S_r}\set{\lambda_1^r(U_r) + \frac{r^{d+\theta}}{tr^{-2}}
 \int_{U_r^c} {\rm d}(x, \partial U_r)^{\theta} \,dx}\label{vp4} 
\end{equation} 
for any $r>0$ under the three assumptions: the first is on the coarse graining step ($\# \S$ is negligible) 
and the second and third are that \eqref{spectral} and \eqref{hole} respectively hold for 
$U \in \S$ in some uniform manners. 
The first one will be verified in Section \ref{MEO} and the second and third ones in Section \ref{patching}. 

Now, if we can find a scale $r=r(t)$ for which the infimum in \eqref{vp4} stays bounded both above and below 
by positive constants as $t \to \infty$, then $tr^{-2}$ gives the asymptotic order of $\log S_t$. 
It might seem natural to take $r=t^{1/(d+\theta+2)}$, to satisfy $r^{d+\theta}/tr^{-2}=1$, at the first sight. 
However, this scale gives a wrong magnitude $t^{(d+\theta)/(d+\theta+2)}$. 
The key to finding the correct scale is that we can easily decrease the value of 
the integral $\int_{U_r^c} {\rm d}(x, \partial U_r)^{\theta} dx$. 
For instance, consider a domain with many tiny holes
\begin{equation}
 U_r^c = (-n,n)^d \setminus \bigcup_{q \in \Z^d}C(\delta(r)q , r^{-1}). \label{punched}
\end{equation}
with $n \in \N$ and $\delta(r) \to 0$ as $r \to \infty$. 
Then we have 
\begin{equation*}
 \int_{U_r^c} {\rm d}(x, \partial U_r)^{\theta} dx \asymp \delta(r)^{\theta},
\end{equation*} 
which goes to 0 as $r \to \infty$. 
Although such holes generally increase the principal eigenvalue $\lambda_1^r(U_r)$, 
it is known that we can take $\delta(r)$ small to some extent while keeping the control of $\lambda_1^r(U_r)$. 
Indeed, Rauch and Taylor \cite{RT75} proved that for this specific example with the hard traps (i.e.\ $h = \infty$), 
\begin{equation}
 \begin{split}
  \delta_c(r)=
  \left\{
  \begin{array}{lr}
   \smash[t]{( \log r )^{-\frac{1}{2}}} &(d=2),\\[5pt]
   \smash[t]{r^{-\frac{d-2}{d}}}  &(d \ge 3),
  \end{array}\right.\label{ccr}
 \end{split}
\end{equation}
are the critical intervals in the following sense: 
\begin{enumerate}
\item{If $\lim_{r \to \infty}\delta(r)/\delta_c(r)=0$, 
then $\lim_{r \to \infty}\lambda_1(U_r) = \infty$.}
\item{If $\lim_{r \to \infty}\delta(r)/\delta_c(r)=\infty$, 
then $\lim_{r \to \infty}\lambda_1(U_r) = \lambda_1(([-n,n]^d)^c)$.}
\end{enumerate}
Therefore, we have to take the scale $r$ at least so large as to satisfy 
$r^{d+\theta}/tr^{-2}=\delta_c(r)^{-\theta}$. 
Otherwise, we find that the infimum in \eqref{vp4} goes to zero as $r \to \infty$, 
by considering the domain \eqref{punched} with a large $n$ and an appropriate $\delta(r)$. 
The next proposition, a generalization of the above criticality, 
shows that the infimum is actually bounded below for this choice of the scale. 
%%%%%%%%%%%%%%%%%%%%%%%%%%%%%%%%%%%%%%%%%%%%%%%Prop 3%%%%%%%%%%%%%%%%%%%%%%%%%%%%%%%%%%%%%%%
\begin{prop}\label{prop3}
 There exists a function $M_2(\e) \to \infty$ $(\e \to 0)$ such that
 if $U_r \subset \S_r$ satisfies
 \begin{equation}
  \# \set{q \in U_r^c \cap \frac{1}{r}\Z^d;\, {\rm d}(q, \partial U_r) \ge \e \delta_c(r) } < \e r^d \label{ass}, 
 \end{equation}
 then $\lambda_1^r(U_r) > M_2(\e)$. 
 In particular, we have 
 \begin{equation*}
  \inf_{r \ge 1,\, U_r \in \S_r}\set{\lambda_1^r(U_r) + 
  \delta_c(r)^{-\theta}\int_{ U_r^c } {\rm d}(x, \partial U_r)^{\theta}dx} > 0. \label{Prop3}
 \end{equation*}
\end{prop}
\begin{proof}
We first recall that the principal eigenvalue can be expressed by the Dirichlet form 
\begin{equation}
 \lambda_1^r(U_r) = \int_{\T_r}\frac{1}{2}|\nabla \psi_r |^2(x) + r^2 1_{U_r}(x)\psi_r^2(x) \, dx\label{RR}
\end{equation}
using the associated $L^2$-normalized eigenfunction $\psi_r$. 
Our basic strategy is estimating the right-hand side by patching local estimates. 
For the local estimates, we use the following lemma. 
%%%%%%%%%%%%%%%%%%%%%%%%%%%%%%%%%%%%%%%%%%%%%%%Lem 2%%%%%%%%%%%%%%%%%%%%%%%%%%%%%%%%%%%%%%%
\begin{lem}\label{lem2}
 There exists $c_1(d)>0$ such that  for any $i \in \Z^d$, 
 $C(y,\frac{1}{r}) \subset C(\e \delta_c(r)i, \,2\e \delta_c(r))$, $\e > 0$, 
 and $\phi \in W^{1,2}(C(\e \delta_c(r)i, \, 2\e \delta_c(r)))$, we have 
 \begin{equation}
  \frac{1}{\| \phi \|_2^2 }\int_{C(\e \delta_c(r)i, \,2\e \delta_c(r))}\frac{1}{2}|\nabla \phi|^2(x) + 
  r^2 1_{C(y, \frac{1}{r}) }(x)\phi^2(x) \, dx \ge c_1(d)\e^{-d}. \label{Lem2}
 \end{equation} 
\end{lem}
\begin{proof}
Using the scaling with the factor $\e \delta_c(r)$, we can bound the right-hand side of 
\eqref{Lem2} below by 
\begin{equation*}
  (\e \delta_c(r))^{-2} \inf_{\phi \in W^{1,2}(C(i, \, 2))}
  \frac{1}{\| \phi \|_2^2 }\int_{C(i, \,2)}\frac{1}{2}|\nabla \phi|^2(x) + 
  (r \e \delta_c(r))^2 1_{C(y, \frac{1}{r \e \delta_c(r)}) }(x)\phi^2(x) \, dx.\label{Lem2-lower} 
\end{equation*}
Note that the infimum appearing in the above expression is 
the Neumann principal eigenvalue of the associated operator.  
The asymptotic behavior of the eigenvalue of this kind of operator has been studied thoroughly 
by Ben-Ari \cite{BA08} (we also refer the reader to Taylor's earlier work \cite{Tay76} for the case $d \ge 3$). 
Our situation can be found in Theorem 1.3 of \cite{BA08}, which tells us 
\begin{equation*}
\begin{split}
  \inf_{\phi \in W^{1,2}(C(i, \, 2))}
  & \frac{1}{\| \phi \|_2^2 }\int_{C(i, \,2)}\frac{1}{2}|\nabla \phi|^2(x) + 
  (r \e \delta_c(r))^2 1_{C(y, \frac{1}{r \e \delta_c(r)}) }(x)\phi^2(x) \, dx\\
  & \sim \left\{
  \begin{array}{lr}
   \smash[t]{ c(2) ( \log (r \e \delta_c(r)) )^{-1}} &(d=2),\\[5pt]
   \smash[t]{ c(d) ( r \e \delta_c(r) )^{2-d}}  &(d \ge 3). 
  \end{array}\right.
\end{split}
\end{equation*}
Recalling the definition of $\delta_c(r)$, \eqref{Lem2} follows immediately. 
\end{proof}
Now we show how to patch the local estimates. 
Let $\e > 0$ be small and $\mathcal{I}(r)$ be the collection of $i \in \Z^d$ for 
which $C(\e \delta_c(r)i,\e \delta_c(r))$ intersects both $U_r$ and $U_r^c$. 
Then, for large $r$, each $C(\e \delta_c(r)i,2\e \delta_c(r))$ $(i \in \mathcal{I}(r))$ contains 
at least one $1/r$-box $\subset U_r$. 
Therefore for all $i \in \mathcal{I}(r)$, we have 
\begin{equation}
 \frac{\int_{C(\e \delta_c(r)i, \,2\e \delta_c(r))}\frac{1}{2}|\nabla \psi_r|^2(x) + 
 r^2 1_{U_r}(x)\psi_r^2(x) \, dx}
 {\int_{C(\e \delta_c(r)i, \,2\e \delta_c(r))}\psi_r^2(x) \, dx} 
 \ge c_1(d)\e^{-d}\label{local}
\end{equation} 
by using Lemma \ref{lem2} with $\phi=\psi_r|_{C(\e \delta_c(r)i, \,2\e \delta_c(r))}$. 
Moreover, since there exists $m(d) \in \N$ such that every $x \in \R^d$ is contained in at most $m(d)$ different 
$C(\e \delta_c(r)i, 2 \e \delta_c(r))$'s, we find 
\begin{equation}
 \begin{split}
  &\int_{\T}\frac{1}{2}|\nabla \psi_r |^2(x) + r^2 1_{U_r}(x)\psi_r^2(x) \, dx\\
  \ge &\,m(d)^{-1} \sum_{i \in \mathcal{I}(r)} \int_{C(\e \delta_c(r)i, \,2\e \delta_c(r))} 
  \frac{1}{2}|\nabla \psi_r |^2(x) + r^2 1_{U_r}(x)\psi_r^2(x) \, dx. \label{num}
 \end{split}
\end{equation}
On the other hand, it is easy to see that 
\begin{equation*}
 \bigcup_{q \in U_r^c \cap \frac{1}{r}\Z^d;\, {\rm d}(q, \partial U_r) < \e \delta_c(r)}C\big(q,\frac{1}{r}\big)
 \subset \bigcup_{i \in \mathcal{I}(r)} C(\e \delta_c(r)i, 2\e \delta_c(r)) 
\end{equation*}
for large $r$. From this and the assumption \eqref{ass}, it follows 
\begin{equation*}
 \left|\T_r \setminus \bigcup_{i \in \mathcal{I}(r)} C(\e \delta_c(r)i, 2\e \delta_c(r))\right| \le \e 
\end{equation*}
when $r$ is sufficiently large. Therefore, 
\begin{equation}
  1=\| \psi _r \|_2^2 
  \le \sum_{i \in \mathcal{I}(r)} \int_{C(\e \delta_c(r)i, \,2\e \delta_c(r))}\psi_r^2(x) \, dx
  + \| \psi_r \|^2_{\infty} \e. \label{den}
\end{equation}
We consider the case $\| \psi_r \|_{\infty} \le \e^{-1/4}$ first. In this case, we have 
\begin{equation*}
  \lambda_1^r(U_r) \ge \frac{m(d)^{-1}c_1(d)\e^{-d}\sum_{i \in \mathcal{I}(r)} 
  \int_{C(\e \delta_c(r)i, \,2\e \delta_c(r))} \psi_r^2(x) \, dx}
  {\sum_{i \in \mathcal{I}(r)} \int_{C(\e \delta_c(r)i, \,2\e \delta_c(r))}\psi_r^2(x) \, dx + \e^{1/2}}
\end{equation*}
by substituting \eqref{num} and \eqref{den} into \eqref{RR} and using \eqref{local}. 
The right-hand side is greater than 
$(2m(d))^{-1}c_1(d)\e^{-d}$ when $\e \le 1/4$. 
Next, we consider the case $\| \psi_r \|_{\infty} > \e^{-1/4}$. This case is easier since 
we know the following $L^{\infty}$-bound for the $L^2$-normalized eigenfunction (see e.g. (3.1.55) of \cite{Szn98})
\begin{equation*}
 \| \psi_r \|_{\infty} \le c_2(d) \lambda^r_1(U_r)^{d/4}, 
\end{equation*}
which gives $\lambda_1^r(U_r) \ge c_2(d)^{-4/d}\e^{-1/d}$. 
Combining the estimates in the two cases, we obtain 
\begin{equation*}
 \lambda^r_1(U_r) \ge ((2m(d))^{-1}c_1(d)\e^{-d}) \wedge (c_2(d)^{-4/d}\e^{-1/d})
\end{equation*}
and the former part of the proposition is proven. 

From the former part, we find 
\begin{equation}
\begin{split}
  &\delta_c(r)^{-\theta}\int_{ U_r^c } {\rm d}(x, \partial U_r)^{\theta}dx \to 0 \quad {\rm as} \quad r \to \infty \\
  &\Longrightarrow \lambda_1^r(U_r) \to \infty \quad {\rm as} \quad r \to \infty \label{dichotomy}
\end{split}
\end{equation}
and the latter part follows immediately.
\end{proof}
This proposition tells us that the correct scale $r$ should be 
\begin{equation}
 \begin{split}
  r=
  \left\{
  \begin{array}{lr}
   {t^{\frac{1}{4+\theta}}( \log t )^{\frac{\theta}{8+2\theta}}} &(d=2),\\[5pt]
   {t^{\frac{d}{d^2+2d+2\theta}}}  &(d \ge 3), 
  \end{array}\right.\label{scale}
 \end{split}
\end{equation}
so that $r^{d+\theta}/tr^{-2} \sim \delta_c(r)^{-\theta}$ as $t \to \infty$ and thus 
\eqref{vp4} becomes 
\begin{equation}
 \log S_t \sim -tr^{-2}\inf_{U_r \in \S_r}\set{\lambda_1^r(U_r) + \delta_c(r)^{-\theta}
 \int_{U_r^c} {\rm d}(x, \partial U_r)^{\theta} \,dx}.\label{vp5}
\end{equation} 
For these scales, $tr^{-2}$ actually gives the correct magnitudes 
\begin{equation*}
 \begin{split}
  tr^{-2}=
  \left\{
  \begin{array}{lr}
   {t^{\frac{2+\theta}{4+\theta}}( \log t )^{-\frac{\theta}{4+\theta}}} &(d=2),\\[5pt]
   {t^{\frac{d^2+2\theta}{d^2+2d+2\theta}}}  &(d \ge 3). 
  \end{array}\right. 
 \end{split}
\end{equation*}
\begin{rem}
As is mentioned before, we have to assure that the infimum in \eqref{vp5} is also bounded above. 
It is possible to prove it here by considering the domain \eqref{punched} with $n=1$ and $\delta(r)=\delta_c(r)$ 
but we postpone the discussion to Appendix \ref{B} since we will have a slightly different 
variational problem after the coarse graining. 
\qed \end{rem}
% 
%%%%%%%%%%%%%%%%%%%%%%%%%%%%%%%%%%%%%%%%%%%%%%%%%%section2.3%%%%%%%%%%%%%%%%%%%%%%%%%%
\subsection{Coarse graining}\label{MEO}
In this section, we give the coarse graining scheme which reduces the combinatorial complexity 
of configurations by replacing dense traps by a large box-shaped hard traps. 
Throughout this section, we are dealing with the scaled traps with the correct scale 
$r$ in \eqref{scale}. The scaled configuration of points 
$\sum_q \delta_{r^{-1}(q+\xi_q)} $ is denoted by $\xi^r$. 

%%%%%%%%%%%%%%%%%%%%%%%%%%%%%%%%%%%%%%%%%%%%%%%%%%definition of the density set%%%%%%%%%%%%%%%
We take a positive number $\eta \in (0,1)$ so small as to satisfy
\begin{equation}
 \eta^2+\Bigl( \frac{d-2}{2}+\frac{\theta}{d} \Bigr)\eta < \frac{\theta}{d} \wedge \frac{1}{2} \label{eta}
\end{equation}
and let 
\begin{equation*}
 \gamma = \frac{d-2}{d}+\frac{2\eta}{d}<1. 
\end{equation*}
We further introduce some notations concerning a dyadic decomposition of $\R^d$. 
Let $\mathcal{I}_k$ be the collection of indices of the form
\begin{equation*}
 \index = (i_0, i_1, \ldots , i_k) \in \Z^d \times (\set{0, 1}^d)^k. 
\end{equation*}
We associate to the above index $\index$ a box: 
\begin{equation*}
 C_{\index} = q_{\index} + 2^{-k}[0,1]^d \;\textrm{ where }\; q_{\index} = i_0 + 2^{-1}i_1 + \cdots +2^{-k}i_k. 
\end{equation*}
For $\index \in \mathcal{I}_k$ and $k' \le k$, we define the truncation 
\begin{equation*}
 [\index]_{k'} = (i_0, i_1, \ldots , i_{k'}). 
\end{equation*}
The notation $\index \preceq \index'$ means that $\index$ is a truncation of $\index'$. 
Finally, we introduce 
\begin{equation*}
 n_{\beta}(r) = \Bigl[ \beta \,\frac{\log r}{\log 2} \Bigr] 
\end{equation*}
for $\beta > 0$ so that $2^{-n_{\beta}-1} < r^{-\beta} \le 2^{-n_{\beta}}$. 

Now we give the precise definition of the ``dense traps'' in the first paragraph.  
%%%%%%%%%%%%%%%%%%%%%%%%%%%%%%%%%%%%%%%%%%%%%%%Def 1%%%%%%%%%%%%%%%%%%%%%%%%%%%%%%%%%%%%%%%
\begin{Def}\label{def1}
 We call $C_q$ $(q \in \Z^d)$ a \underline{\smash[b]{density}} box if all 
 $C_{\index}$'s $(\index \in \mathcal{I}_{n_{\eta\gamma}}$, $q \preceq \index)$ satisfy the following: 
 \begin{equation}
 \begin{split}
  &\textit{for at least half of }\index' \succeq \index \;(\index' \in \mathcal{I}_{n_{\gamma}}), \\
  &q_{\index'} + 2^{-n_{\gamma}-1}[0,1]^d \textit{ contains a point of }\xi^r.\label{myden}
 \end{split}
 \end{equation}
 The union of all \underline{\smash[b]{density}} boxes is denoted by $\smash[b]{\underline{\den}}_r(\xi)$. 
\end{Def}
\begin{rem}
We use $2^{-n_{\gamma}-1}$ instead of $2^{-n_{\gamma}}$ in the definition 
to have \emph{separated} traps. The role of this choice will be clear in the proof 
of Proposition \ref{prop4} (see \eqref{each}). 
\qed \end{rem}
\noindent
In \cite{Szn98}, Sznitman defined density boxes in a different way and proved that they can be replaced by hard traps. 
We shall prove that our \underline{\smash[b]{density}} set is a subset of Sznitman's one 
to use the result in \cite{Szn98}. 
We start by recalling Sznitman's definition of the density set and a result on the principal eigenvalue. 
For $\index \in \mathcal{I}_k$, the skeleton of the traps is defined by 
\begin{equation*}
 K_{\index} = 2^k \bigg( \bigcup_{x \in C_{\index}\cap \,{\rm supp}\,\xi^r} \overline{B}(x, \sqrt{d}/r) \bigg). 
\end{equation*} 
Sznitman defined the density box as follows:
%%%%%%%%%%%%%%%%%%%%%%%%%%%%%%%%%%%%%%%%%%%%%%%Def 2%%%%%%%%%%%%%%%%%%%%%%%%%%%%%%%%%%%%%%%
\begin{Def}\label{def2} {\rm (pp.\ 150-152} {\rm in} {\rm \cite{Szn98})}
 $C_{\index}$ $(\index \in \mathcal{I}_{n_{\gamma}})$ is called a density box if the 
 quantitative Wiener criterion:  
 \begin{equation}
  \sum_{1 \le k \le n_{\gamma}} {\rm cap}(K_{[\index]_k}) \ge \delta n_{\gamma} \label{qWc}
 \end{equation}
 holds for some $\delta > 0$. 
 Here ${\rm cap}(\,\cdot\,)$ denotes the capacity relative to $1-\Delta/2$ when $d=2$ 
 and $-\Delta/2$ when $d \ge 3$. 
 The union of all density boxes is denoted by ${\den}_r(\xi)$. 
\end{Def}
\noindent
The next theorem enables us to replace the density boxes by hard traps 
without inducing a substantial upward shift of the principal eigenvalue.  
\begin{spec} {\rm (Theorem 4.2.3} {\rm in} {\rm \cite{Szn98})} 
There exists $\rho>0$ such that for all $M>0$ and sufficiently large $r$, 
\begin{equation}
 \sup_{\xi \in \Xi}
 \left(\lambda^{r}_{1}\left(r^{-1}{\rm supp}\,V(\,\cdot\,,\xi), 
 {\mathcal{R}}_r(\xi) \right)\wedge M 
 -\lambda^{r}_{1}\left(r^{-1}{\rm supp}\,V(\,\cdot\,,\xi)\right)\wedge M \right) \le r^{-\rho},\label{spec}
\end{equation}
where ${\mathcal{R}}_r(\xi)=\T_r \setminus {\mathcal{D}}_r(\xi)$ 
and $\lambda^{r}_{1}(U, R)$ denotes the principal eigenvalue 
of Dirichlet-Schr\"{o}dinger operator $-1/2\Delta + r^2\cdot 1_U$ in $R$. 
\end{spec}
\noindent
As is announced before, we show the next proposition to apply this theorem 
to our \underline{\smash[b]{density}} set. 
%%%%%%%%%%%%%%%%%%%%%%%%%%%%%%%%%%%%%%%%%%%%%%%Prop 4%%%%%%%%%%%%%%%%%%%%%%%%%%%%%%%%%%%%%%%%%%%%
\begin{prop}\label{prop4}
 $\myden \subset \den_r(\xi).$ Accordingly, $\myrare \stackrel{ \rm def}{=} \T_r \setminus \myden \supset {\mathcal{R}}_r(\xi)$
\end{prop}
\begin{proof}
Let $C_q$ be a \underline{\smash[b]{density}} box. We check the quantitative Wiener criterion \eqref{qWc}
for all $\index \succeq q$ ($\index \in \mathcal{I}_{n_{\gamma}}$) by showing
\begin{equation}
 {\rm cap}(K_{[\index]_k}) \ge c_3(d) \quad \textrm{for all} \quad k \le n_{\eta \gamma}. \label{termwise}
\end{equation}
To get the lower bound for the capacity, we use the following variational characterization: 
\begin{equation*}
 {\rm cap}(K) = \sup \bigg\{\bigg( \iint g(x,y) \,\nu(dx) \,\nu(dy) \bigg)^{-1}; \,
 \nu \in \mathcal{M}_1(K)\bigg\}, 
\end{equation*}
where $\mathcal{M}_1(K)$ denotes the set of probability measure supported on $K$ 
and $g(\, \cdot \, , \, \cdot \,)$ the Green function 
corresponding to $1-\Delta/2$ when $d=2$ and to $-\Delta/2$ when $d \ge 3$. 
By this expression, the proof of \eqref{termwise} is reduced to finding a 
$\nu_k \in \mathcal{M}_1(K_{[\index]_k})$ which satisfies 
\begin{equation}
 \iint g(x,y) \,\nu_k(dx) \,\nu_k(dy) \le c_3(d)^{-1} \label{ener}
\end{equation}
for each $k \le n_{\eta \gamma}$. 

Now, note that \eqref{myden} remains valid for $[\index]_k$ instead of $\index \in \mathcal{I}_{n_{\eta \gamma}}$ 
as long as $k \le n_{\eta \gamma}$. 
Therefore for such $k$, we can find a collection of points 
\begin{equation*}
 \{ x_{m} \in q_{\index_m}+2^{-n_{\gamma}-1}[0,1]^d ; \, \index_m \in \mathcal{I}_{n_{\gamma}} \textrm{ are distinct.} \}_{1 \le m \le n}
 \subset {\rm supp}\, \xi^r 
\end{equation*}
whose cardinality $n \ge 2^{d(n_{\gamma}-k)-1}$. 
We denote by $e_m$ and ${\rm cap}_m$ respectively the equilibrium measure and the capacity of 
$2^k\overline{B}(x_m, \sqrt{d}/r)$ and let
\begin{equation*}
 \nu_k = \frac{\sum_{m=1}^n e_m}{\sum_{m=1}^n {\rm cap}_m} \in \mathcal{M}_1(K_{[\index]_k}). 
\end{equation*}
Let us show that this $\nu_k$ satisfies \eqref{ener}. We use the fact $\iint g(x,y) \, e_m(dx) \,e_m(dy) = {\rm cap}_m$ 
to obtain 
\begin{equation}
 \begin{split}
  &\,\iint g(x,y) \,\nu_k(dx) \,\nu_k(dy)\\
  = &\, \biggl( {\sum_{m=1}^n {\rm cap}_m} \biggr)^{-2} \biggl(\sum_{m = 1}^n \iint g(x,y) \,e_m(dx) \,e_m(dy)
  + \sum_{l \neq m} \iint g(x,y) \, e_l(dx) \,e_m(dy) \biggr)\\
  \le &\, \biggl( {\sum_{m=1}^n {\rm cap}_m} \biggr)^{-1} + {\rm const}(d) \iint_{(0,1)^d \times (0,1)^d} g(x,y) \, dx \, dy.
  \label{each} 
 \end{split}
\end{equation}
In the second inequality, we have implicitly used the fact that 
${\rm d}({\rm supp}\,e_l, {\rm supp}\,e_m) \ge 2^{-n_{\gamma}-2}$ for sufficiently large $r$, 
which is due to our definition of the \underline{\smash[b]{density}} set, 
to replace the sum $\sum_{l \neq m}$ by the integral. 
Since the last integral in \eqref{each} is a constant depending only on $d$, 
it suffices for \eqref{ener} to show that 
${\sum_{m=1}^n {\rm cap}_m} \to \infty$ $(r \to \infty)$. 
If we recall that ${\rm cap}_m$ is just the capacity of a ball with radius $2^{k} \sqrt{d}/r $, 
we find 
\begin{equation*}
 \begin{split}
  \sum_{m=1}^n {\rm cap}_m \ge 
  \left\{
  \begin{array}{lr}
   c_4(d=2) \bigl(\log (2^{-k} r)\bigr)^{-1} 2^{d(n_{\gamma}-k)-1} &(d=2),\\[8pt]
   c_4(d) (2^{k} /r)^{d-2} 2^{d(n_{\gamma}-k)-1} &(d \ge 3).
  \end{array}\right.
  \end{split}
\end{equation*}
When $d \ge 3$ and $1 \le k \le n_{\eta\gamma}$, the right-hand side is larger than
\begin{equation*}
 \begin{split}
  c_4(d) r^{2-d} 2^{dn_{\gamma}-2k-1} &\ge c_4(d) r^{2-d+d\gamma-2\eta\gamma}/8 \\
  & = c_4(d) r^{2\eta(1-\gamma)}/8 \\
  &\to \infty \qquad (r \to \infty),
 \end{split}
\end{equation*}
as desired. 
Here we have used $2^{-n_{\beta}-1} < r^{-\beta} \le 2^{-n_{\beta}}$ for $\beta > 0$ in the first inequality. 
The case $d=2$ can be treated by the same way and the proof of Proposition \ref{prop4} is completed. 
\end{proof}
%%%%%%%%%%%%%%%%%%%%%%%%%%%%%%%%%%%%%%%%%%%%%%%%%%estimate on the density set%%%%%%%%%%%%%%%%%%%%%%%%%%%%%%%%
Now we turn on to the estimate for the number of non-\underline{\smash[b]{density}} boxes in $\T_r$. 
It is clear from the definition that the number should be very small. 
However, we need a quantitative estimate for the coarse graining to go well. 
We pick a positive parameter 
\begin{equation}
 \chi \in \Bigl( 2\eta^2+\Bigl( d-2+ \frac{2\theta}{d} \Bigr)\eta, \frac{2\theta}{d} \wedge 1 \Bigr)
 \label{chi}
\end{equation}
so that 
\begin{gather}
 d(1-\eta\gamma)+(1-\gamma)\theta+\chi > d+\frac{2\theta}{d},\label{prob} \\
 d+\chi < d+\frac{2\theta}{d}. \label{card} 
\end{gather} 
It is easy to see from \eqref{eta} that such a choice of $\chi$ is possible. 
Thanks to the relation \eqref{prob}, the right-hand side of the next proposition is
\begin{equation*}
 o\bigl( \exp\bigl\{- r^{d+\frac{2\theta}{d}} \bigr\} \bigr) 
 = o\Bigl( \exp\Bigl\{- t^{\frac{d^2+2\theta}{d^2+2d+2\theta}}\Bigr\} \Bigr)
\end{equation*}
%%%%%%%%%%%%%%%%%%%%%%%%%%%%%%%%%%%%%%%%%%%%%%%Prop 5%%%%%%%%%%%%%%%%%%%%%%%%%%%%%%%%%%%%%%%%%%%%
\begin{prop}\label{prop5}
 \begin{equation*}
  \P_{\theta}(|\myrare| \ge r^{\chi})
  \le \exp \set{-c_5(d)r^{d(1-\eta\gamma)+(1-\gamma)\theta+\chi}}.\label{Prop5} 
 \end{equation*}
\end{prop}
\begin{proof}
Throughout the proof, $c_5(d)>0$ is a constant whose value may change line by line. 
We start with an estimate for the probability of $C_q \not\subset \myden$. 
To this end, we consider the following necessary condition: 
\begin{equation}
 \begin{split}
  &\textrm{there exists an }\index \succeq q \; (\index \in \mathcal{I}_{n_{\eta\gamma}}) \textrm{ such that for a half of }
  \index' \succeq \index \: (\index' \in \mathcal{I}_{n_{\gamma}}), \\
  &r^{-1}q'+r^{-1}\xi_{q'} \notin q_{\index'} + 2^{-n_{\gamma}-1}[0,1]^d \textrm{ for}
  \textrm{ all } r^{-1}q' \in q_{\index'} + 2^{-n_{\gamma}-1}[0,1]^d. \label{suff}
 \end{split}
\end{equation}
Note that the events in the second line are independent in $\index' \in \mathcal{I}_{n_{\gamma}}$. 
Moreover, the probability of the each event is 
\begin{equation*}
\begin{split}
 \P_{\theta}(&q'+\xi_{q'} \notin r(q_{\index'} + 2^{-n_{\gamma}-1}[0,1]^d) 
 \textrm{ for all } q' \in r(q_{\index'} + 2^{-n_{\gamma}-1}[0,1]^d))\\
 \le &\exp\Bigl\{- \int_{r^{1-\gamma}[0,1]^d}{\rm d}(x, \partial (r^{1-\gamma}[0,1]^d))^{\theta}dx (1+o(1))\Bigr\}\\
 \le &\exp\set{-c_5(d)r^{(1-\gamma)(d+\theta)}}, 
\end{split}
\end{equation*}
where we have used \eqref{Lem1-upper} for the first inequality. 
Therefore, summing over the choices of the indices $\index$ and $\index'$'s in \eqref{suff}, we obtain 
\begin{equation*}
 \begin{split}
  &\P_{\theta}(C_q \not\subset \myden) \\
  \le &\, 2^{d n_{\eta \gamma}} \binom{2^{d(n_{\gamma}- n_{\eta \gamma})}}{2^{d(n_{\gamma}- n_{\eta \gamma})-1}}
  \exp \set{-c_5(d)r^{(1-\gamma)(d+\theta)}}^{2^{d(n_{\gamma}- n_{\eta \gamma})-1}}\\
  \le &\, \exp \set{-c_5(d) r^{(1-\gamma)(d+\theta) + d\gamma(1-\eta)}} 
 \end{split}
\end{equation*}
for large $r$. In the second line, the first factor represents the choice of the index $\index$ and  
the second factor the choice of the indices $\index'$'s. 
Since the event \eqref{suff} itself is independent in $q \in \Z^d$, we have 
\begin{equation*}
 \begin{split}
  \P_{\theta}(|\T_r \setminus \myden| \ge r^{\chi})
  \le &\,t^{d r^{\chi}}\left(\exp \set{-c_5(d)r^{(1-\gamma)(d+\theta) + d\gamma(1-\eta)}}\right)^{r^{\chi}}\\ 
  \le &\,\exp \set{-c_5(d)r^{d(1-\eta\gamma)+(1-\gamma)\theta+\chi}}, 
 \end{split}
\end{equation*}
which is the desired estimate. 
\end{proof}
\noindent
Now, let us bound the cardinality of 
\begin{equation*}
 \begin{split}
  \smash[b]{\underline{\S}}_r = \bigl\{ & \left( \myrare, r^{-1}{\rm supp}\,V(\,\cdot\,, \xi) \cap \myrare \right);\\
  &\,\xi \in \Xi, \myrare \neq \emptyset \textrm{ is connected}, |\myrare| < r^{\chi}\bigr\}. 
 \end{split}
\end{equation*}
To this end, we first note that $|\myrare|$, the number of unit cubes contained in $\myrare$, 
varies from 1 to $r^{\chi}$. 
Secondly, we have at most $(t/r)^{d r^{\chi}}$ choices for the configuration of 
the unit cubes in $\myrare$, for any given $|\myrare|<r^{\chi}$. 
Finally, there are at most $2^{r^{d}}$ possible configurations of the traps 
inside each unit cube in $\myrare$. 
Therefore, we have 
\begin{equation}
 \begin{split}
  \# \smash[b]{\underline{\S}}_r &\le r^{\chi} (t/r)^{d r^{\chi}}(2^{r^d})^{r^{\chi}}\\
  &=\, \exp\set{r^{d+\chi}\log 2 (1+o(1))}\\
  &=\exp\Bigl\{o\Bigl( t^{\frac{d^2+2\theta}{d^2+2d+2\theta}} (\log t)^{-\frac{\theta}{4+\theta}} \Bigr)\Bigr\},
  \label{complexity} 
 \end{split}
\end{equation} 
where the third line comes from the relation \eqref{card}. 
%%%%%%%%%%%%%%%%%%%%%%%%%%%%%%%%%%%%%%%%%%%%%%%%%%section2.4%%%%%%%%%%%%%%%%%%%%%%%%%%%%%%
\subsection{Patching estimates}\label{patching}
We complete the proof of survival asymptotics in this section. 
Throughout this section, we use the correct scale $r$ in \eqref{scale} and let
$\e>0$ denote an arbitrary small number. 
We introduce 
\begin{equation*}
 M_r = \inf_{(R_r,\, U_r) \in \smash[b]{\underline{\S}}_r}\Bigg\{\lambda_1^r(U_r, R_r)
 + \delta_c(r)^{-\theta}\int_{ R_r \setminus U_r } {\rm d}(x, \partial (R_r \setminus U_r))^{\theta}dx \Bigg\} \label{M}
\end{equation*}
to describe the asymptotics. 
We know $\inf_{r \ge 1} M_r > 0$ from Proposition \ref{prop3} and 
we can also prove $\sup_{r \ge 1} M_r < \infty$ by substituting  
the punched domain \eqref{punched} with $n=1$ and $\delta(r)=\delta_c(r)$ to $R_r \setminus U_r$. 
We postpone the proof of the latter fact to Appendix \ref{B}. 

What we prove here is the following asymptotics which is finer than the results stated in Section \ref{intro}. 
%%%%%%%%%%%%%%%%%%%%%%%%%%%%%%%%%%%%%%%%%%%%%%%Thm 3%%%%%%%%%%%%%%%%%%%%%%%%%%%%%%%%%%%%%%%%%%%%%
\begin{thm}\label{thm3}
Let $r$ be as in \eqref{scale}. For modified potential \eqref{mptl} with $\e, L, h=1$, we have 
\begin{equation*}
 \frac{1}{tr^{-2}}\log S_t \sim - M_r\label{fine} \quad {\rm as} \quad t \to \infty.
\end{equation*} 
\end{thm}
\begin{rem}
The extensions of this theorem for other values of $\e, L, h$ are routine with appropriate changes on the notation. 
Though we have this finer result only for the modified traps \eqref{mptl}, it seems not so far from 
the original model at least in the case of hard traps. Indeed, for the hard traps, 
the modification is equivalent to discretizing the distribution of $\xi_q$ as 
\begin{equation*}
 \P_{\theta}(\xi_q \in dx) = N_{\rm disc}(d,\theta) \sum_{q \in \Z^d} \exp\{-|q|^{\theta} \}\delta_q(dx). 
\end{equation*}
However, we still do not know whether $\limsup_{r \to \infty}M_r$ and $\liminf_{r \to \infty}M_r$ coincide 
or not. 
\qed \end{rem}
\noindent
\textit{Upper bound}: 
For any $U \subset \R^d$, we have the following upper bound on the Feynman-Kac semigroup 
(see e.g.\ (3.1.9) of \cite{Szn98}): 
\begin{equation*}
 \begin{split}
  & \sup_{x \in \R^d}E_x\left[ \exp\set{-\int_0^t 1_U (B_s) \,ds};\,T_{\T} >t \right]\\
  & \le c(d)(1+(\lambda_1(U,\T) t)^{d/2})\exp\set{-\lambda_1(U,\T) t}. 
 \end{split}
\end{equation*}
It follows from this estimate that 
\begin{equation*}
 E_0\left[ \exp\set{-\int_0^t 1_U(B_s) \,ds};\,T_{\T} >t \right]
 \le c(d, \e)\exp\set{-(1-\e)\lambda_1(U,\T) t}, \label{unif-up}
\end{equation*}
where $c(d, \e)=\sup_{\lambda > 0}c(d)(1+\lambda^{d/2})\exp\set{-\e\lambda}$. 
Thus, using Spectral control \eqref{spec} and Proposition \ref{prop5}, we have 
\begin{equation}
 \begin{split}
  S_t \le &\, c(d, \e)\E_{\theta} \left[\exp \set{- (1-\e) \lambda_1({\rm supp}\,V(\,\cdot\,, \xi))t }\right]\\
  \le &\, c(d, \e)\E_{\theta} \big[\exp \set{- (1-\e) (\lambda_1^r(r^{-1}{\rm supp}\,V(\,\cdot\,, \xi), \myrare)
  \wedge M_r-r^{-\rho})tr^{-2}};\\
  &\qquad |\myrare| < r^{\chi} \big]
  + \P_{\theta}(|\myrare| \ge r^{\chi})\\
  \le &\, c(d, \e)\# \smash[b]{\underline{\S}}_r  \sup_{(R_r,\, U_r) \in \smash[b]{\underline{\S}}_r} 
  \P_{\theta}(\xi^r(R_r \setminus U_r)=0) \\
  &\times \exp \set{- (1-\e) (\lambda_1^r(U_r, R_r)\wedge M_r-r^{-\rho})tr^{-2}} 
  + o\Bigl( \exp\Bigl\{- t^{\frac{d^2+2\theta}{d^2+2d+2\theta}} \Bigr\} \Bigr) 
  \label{ub}
 \end{split}
\end{equation}
for large $t$, 
where we have used the fact that the principal eigenvalue is the infimum of those over 
the connected components of the domain to assume $R_r$ to be connected. 
Since the factor $\# \smash[b]{\underline{\S}}_r$ (by \eqref{complexity}) 
and the second term is negligible compared with the results, we focus on the variational problem. 
In order to apply Lemma \ref{lem1} to the emptiness probability term, we see that $r(R_r \setminus U_r)$ 
satisfies the assumption of the lemma when $(R_r, U_r)\in \smash[b]{\underline{\S}}_r$. 
%%%%%%%%%%%%%%%%%%%%%%%%%%%%%%%%%%%%%%%%%%%%%%%Prop 6%%%%%%%%%%%%%%%%%%%%%%%%%%%%%%%%%%%%%%%%%%%%%%
\begin{prop}\label{prop6}
 For any $(R_r, U_r) \in \smash[b]{\underline{\S}}_r$, let $W_r = R_r \setminus U_r$. Then we have 
 \begin{equation*}
  \int_{ W_r } {\rm d}(x, \partial W_r)^{\theta} dx
  \ge c_6(d,\theta) r^{- \gamma(\theta + d\eta)} |W_r|
 \end{equation*}
 for large $r$. In particular 
 \begin{equation*}
  \lim_{r \to \infty} \frac{\int_{ rW_r } {\rm d}(x, \partial (rW_r))^{\theta} dx}{|rW_r|} = \infty. 
 \end{equation*}
\end{prop}
\begin{proof}
By the definition of the \underline{\smash[b]{density}} box, each $C_q \subset R_r$ contains a 
$C_{\index}$ ($\index \in \mathcal{I}_{n_{\eta\gamma}}$) such that 
half of $\{C_{\index'}' \stackrel{\rm def}{=} q_{\index'}+2^{-n_{\gamma}-1}[1/4,3/4]^d\}_{\index \preceq \index' \in \mathcal{I}_{n_{\gamma}}}$ 
do not intersect with $U_r$ for large $r$. 
Therefore, the number of such $q_{\index'}+2^{-n_{\gamma}-1}[1/4,3/4]^d$ 
in the whole $R_r$ is larger than $2^{-d-1}2^{dn_{\gamma}-dn_{\eta\gamma}}|R_r|$. 
Since ${\rm d}(x, \partial W_r) \ge {\rm d}(x, \partial C_{\index'}')$ for $x \in C_{\index'}'$, 
we can obtain the desired estimate as follows: 
\begin{equation*}
 \begin{split}
  \int_{ W_r } {\rm d}(x, \partial W_r)^{\theta} dx
  \ge &\, 2^{-d-1}2^{dn_{\gamma}-dn_{\eta\gamma}}|R_r| 
  \int_{ C_{\index'}' } {\rm d}(x, \partial C_{\index'}')^{\theta} dx\\
  \ge &\, c_6(d, \theta) r^{d\gamma(1-\eta)}r^{-\gamma(d+\theta)}|W_r|, 
 \end{split}
\end{equation*}
where we have used the change of variables 
\begin{equation*}
\int_{ C_{\index'}' } {\rm d}(x, \partial C_{\index'}')^{\theta} dx
=(2^{n_{\gamma}-1})^{d+\theta} \int_{ [1/4, 3/4]^d } {\rm d}(x, \partial ([1/4, 3/4]^d) )^{\theta} dx
\end{equation*}
for the second inequality. 
The latter claim follows immediately since \eqref{eta} implies $\theta > \gamma(\theta + d\eta)$. 
\end{proof}
\noindent
Using the relation $tr^{-2}=r^{d+\theta}\delta_c(r)^{\theta}$, Lemma \ref{lem1}, and Proposition \ref{prop6} 
in \eqref{ub}, we obtain 
\begin{equation*}
 \begin{split}
  \frac{1}{tr^{-2}}\log S_t
  \le & \, - (1-\e) \inf_{(R_r,\, U_r) \in \smash[b]{\underline{\S}}_r}\bigg\{\lambda_1^r(U_r, R_r)\wedge M_r-r^{-\rho}\\ 
  & \, + \delta_c(r)^{-\theta}\int_{R_r\setminus U_r} {\rm d}(x, \partial (R_r\setminus U_r))^{\theta} \,dx
  (1+o(1)) \bigg\}\\
  \le & \, - (1-2\e) M_r \label{fin} 
 \end{split}
\end{equation*}
for sufficiently large $r$. Since $\e > 0$ is arbitrary, the proof of the upper bound is completed. \qed \\[10pt]
\textit{Lower bound}: We start with the following obvious bound: 
\begin{equation}
\begin{split}
 \log S_t \ge &\, \log \sup_{(R_r,\,U_r) \in \smash[b]{\underline{\S}}_r'} 
 \P_{\theta}(\xi^r(R_r \setminus U_r)=0)\\
 & \times E_0\Bigl[\exp \Bigl\{-\int_0^{tr^{-2}} r^2 \cdot 1_{U_r}(B_s) \,ds\Bigr\} ; T_{R_r} > tr^{-2} \Bigr], \label{lb'}
\end{split}
\end{equation}
where 
\begin{equation*}
 \smash[b]{\underline{\S}}_r'  = \set{(r^{-1}q+R_r,\,r^{-1}q+U_r); 
 q \in \Z^d, (R_r,\,U_r) \in \smash[b]{\underline{\S}}_r}. 
\end{equation*}
The role of this extension of $\smash[b]{\underline{\S}}_r$ will be clear in the proof 
of Proposition \ref{prop8}. 
We first rewrite the emptiness probability term in the right-hand side of \eqref{lb'}. 
Thanks to Proposition \ref{prop6}, we can use Lemma \ref{lem1} to obtain 
\begin{equation}
\begin{split}
 \log S_t \ge &\, \log \sup_{(R_r,\,U_r) \in \smash[b]{\underline{\S}}_r'} 
 \exp\Bigl\{-r^{d+\theta} \int_{R_r \setminus U_r} {\rm d}(x, \partial (R_r \setminus U_r))^{\theta}dx(1+o(1))\Bigr\}\\
 & \times E_0\Bigl[\exp \Bigl\{-\int_0^{tr^{-2}} r^2 \cdot 1_{U_r}(B_s) \,ds\Bigr\} ; T_{R_r} > tr^{-2} \Bigr]. 
 \label{lb''}
\end{split}
\end{equation}
Next, we rewrite the Brownian motion part of \eqref{lb''}. Though the result seems to be natural, 
the proof is rather complicated. 
%%%%%%%%%%%%%%%%%%%%%%%%%%%%%%%%%%%%%%%%%%%%%%%Prop 8%%%%%%%%%%%%%%%%%%%%%%%%%%%%%%%%%%%%%%%%%%%%%
\begin{prop}\label{prop8}
For sufficiently large $t$, we have
\begin{equation}
 \frac{1}{tr^{-2}}\log S_t \ge - (1+\e) \inf_{(R_r, \,U_r) \in \smash[b]{\underline{\S}}_r'}\Bigg\{\lambda_1^r(U_r, R_r) 
  + \delta_c(r)^{-\theta}\int_{R_r \setminus U_r} {\rm d}(x, \partial (R_r \setminus U_r))^{\theta} dx\Bigg\}. \label{Prop8}
\end{equation} 
\end{prop}
\begin{proof}
We give the proof for the case $h=\infty$ since we need a modified potential with 
$h=\infty$ to derive the lower bound for the original potential in Theorem \ref{thm1}. 
In this case, the second factor in the right-hand side of \eqref{lb''} is $P_0(T_{R_r \setminus U_r} > tr^{-2})$. 
The proof for the case $h<\infty$ is different but much simpler. 
We shall later explain how to adapt the following argument to that case (see Remark \ref{adapt} below). 

Note first that the functional in the infimum in \eqref{Prop8} is invariant under $r^{-1}\Z^d$-shift. 
If we also recall that $\smash[b]{\underline{\S}}_r'$ contains only finite pairs of sets modulo $r^{-1}\Z^d$-shift, 
it follows that there exists $(R_r^*, U_r^*) \in \smash[b]{\underline{\S}}_r'$ which attains the infimum. 
We denote by $p_{*}(t, x, y)$ the transition kernel of the killed Brownian motion when exiting
$R_r^* \setminus U_r^*$ and 
by $\phi_{*}$ the $L^1$-normalized positive eigenfunction corresponding to $\lambda_1^r(U_r^*, R_r^*)$. 
Since ${\rm supp}\, \phi_* \subset R_r^* $, there exists a box $C(r^{-1}q, r^{-1})$ where
\begin{equation}
 \int_{C(r^{-1}q,\, r^{-1})} \phi_{*}(x) \, dx \ge r^{-d-\chi}. \label{mass}
\end{equation}
We can assume $q=0$ by the shift invariance and the extension of $\smash[b]{\underline{\S}}_r$ 
to $\smash[b]{\underline{\S}}_r'$. 
Then, it follows that $C(0,r^{-1}) \subset R_r^* \setminus U_r^*$. 
We also have the following uniform upper bound on $\| \phi_{*} \|_{\infty}$.
\begin{lem}\label{unif}
\begin{equation*}
 \| \phi_{*} \|_{\infty} \le \exp\Bigl\{2\sup_{r \ge 1} M_r\Bigr\} < \infty
\end{equation*}
\end{lem}
\begin{proof}
Since $p_{*}(t, x, y)$ is smaller than the usual heat kernel 
\begin{equation*}
 p(t,x,y)=\frac{1}{(2 \pi t)^{{d}/{2}}}\exp\set{-\frac{|x-y|^2}{2t}}, 
\end{equation*}
we have $p_{*}(1, \,\cdot\,, \,\cdot\,) < 1$ and therefore
\begin{equation*}
\begin{split}
 \phi_{*}(x) &= \,\exp\set{\lambda_1^r (U_r^*, R_r^*) } \int_{R_r^*}  p_{*}(1, x, y) \phi_{*}(y) \,dy\\
 &< \,\exp\set{\lambda_1^r (U_r^*, R_r^*) } \int_{R_r^*}  \phi_{*}(y) \,dy 
\end{split}
\end{equation*}
for all $x \in R_r^*$. 
The rest is easy using $\sup_{r \ge 1} M_r < \infty$ and $\| \phi_{*} \|_1 =1$. 
\end{proof}

From this lemma and the fact $\chi<1$ in \eqref{chi}, we see that the integral in \eqref{mass} 
is not supported on the $r^{-2}$-neighborhood of the boundary: 
\begin{equation*}
\begin{split}
 \int_{C(0,\, r^{-1}) \setminus C(0,\, r^{-1}-r^{-2})} \phi_{*}(x) \, dx 
 & \le \| \phi_{*} \|_{\infty}\, \bigl|C(0,\, r^{-1}) \setminus C(0,\, r^{-1}-r^{-2})\bigr|\\ 
 & =o(r^{-d-\chi}).
\end{split}
\end{equation*}
Therefore, we can discard it to find 
\begin{equation}
 \int_{ C(0,\, r^{-1}-r^{-2})} \phi_{*}(x) \, dx \ge \frac{1}{2} r^{-d-\chi}. \label{mass'}
\end{equation}

Now, let $p_{C}(t, x, y)$ denote the transition kernel of the killed Brownian motion when exiting
$C(0,\, r^{-1})$. Clearly $p_C(t,x,y) \le p_{*}(t,x,y)$ and we can also show that 
\begin{equation*}
 \inf_{x \in C(0,\, r^{-1}-r^{-2})} p_C(r^{-4}, 0, x)
 \ge {\rm const}(d) r^{-2d} \exp\set{-\sqrt{d} r^2 /4} 
\end{equation*}
by using Theorem 2 and (6) in \cite{vdB92}. 
From these estimates for $p_C$ and the Chapman-Kolmogorov identity, we obtain
\begin{equation}
 \begin{split}
 &P_0\left( T_{R_r^* \setminus U_r^*} > tr^{-2} \right)\\
 \ge &\int_{C(0,\, r^{-1}-r^{-2})} p_{C}(r^{-4}, 0, x) \int_{R_r^*} p_{*} (tr^{-2}-r^{-4},x,y) 
 \frac{\phi_{*}(y)}{\| \phi_{*} \|_{\infty}}\,dy \,dx\\
 \ge &\, \| \phi_{*} \|_{\infty}^{-1}\inf_{x \in C(0,\, r^{-1}-r^{-2})} p_{C} (r^{-4}, 0, x)\\
 & \times \exp\set{-\lambda_1^r (U_r^*, R_r^*) tr^{-2}} \int_{C(0,\, r^{-1}-r^{-2})} \phi_{*}(x) \,dx\\
 \ge &\, {\rm const}(d, \theta) r^{-d-\chi-2d}
 \exp\set{-\lambda_1^r (U_r^*, R_r^*) tr^{-2}-\sqrt{d} r^2 /4}, 
 \label{spec*}
 \end{split}
\end{equation}
where we have used Lemma \ref{unif} and \eqref{mass'} in the last inequality. 

Finally, substituting \eqref{spec*} to \eqref{lb''} and recalling 
$r^{d+\theta}/tr^{-2}=\delta_c(r)^{-\theta}$ and $r^2 = o(tr^{-2})$, we find
\begin{equation*}
 \frac{1}{tr^{-2}}\log S_t 
 \ge - (1+\e)\set{\lambda_1^r(U_r^*, R_r^*) + 
 \delta_c(r)^{-\theta}\int_{R_r^* \setminus U_r^*} {\rm d}(x, \partial (R_r^* \setminus U_r^*))^{\theta} dx  }
\end{equation*}
for sufficiently large $r$. This completes the proof of Proposition \ref{prop8}.
\end{proof}
\begin{rem}\label{adapt}
 For a modified potential with $h < \infty$, we need not discard the neighborhood of 
 $\partial C(0, r^{-1})$ and can proceed to \eqref{spec*} directly after Lemma \ref{unif}. 
 Then, the same argument works with $C(0,r^{-1}-r^{-2})$ and $p_c(r^{-4},0,x)$ replaced by 
 $C(0, r^{-1})$ and $e^{-h}p(r^{-2}, 0,x) \le p_{*}(r^{-2},0,x)$, respectively. \qed
\end{rem}
\noindent
Now, note that $\smash[b]{\underline{\S}}_r'$ in the right-hand side of \eqref{Prop8} 
can be replaced by $\smash[b]{\underline{\S}}_r$ since both terms in the infimum are invariant 
under $r^{-1}\Z^d$-shift. 
Therefore, the right-hand side of \eqref{Prop8} equals $-(1+\e)M_r$ and the proof of the 
lower bound is completed.  \qed \\

%%%%%%%%%%%%%%%%%%%%%%%%%%%%%%%%%%%%%%%%%%%%%%%%%%section3%%%%%%%%%%%%%%%%%%%%%%%%%%%%%%
\section{Applications}\label{applications}
%%%%%%%%%%%%%%%%%%%%%%%%%%%%%%%%%%%%%%%%%%%%%%%%%%section3.1%%%%%%%%%%%%%%%%%%%%%%%%%%%%%%
\subsection{Lifshitz tail}\label{Lifshitz}
In this section, we discuss the asymptotic behavior of the density of states of $H_{\xi}$ 
that is defined by the thermodynamic limit 
\begin{equation*}
 \ell(d\lambda) = \lim_{N \to \infty}\frac{1}{(2N)^d} \sum_{i \ge 1}
 \delta_{\lambda_i^D(H_{\xi}\textrm{ in }(-N,N)^d)}(d\lambda). 
\end{equation*}
Here $\lambda_i^D(H_{\xi}\textrm{ in }(-N,N)^d)$ is the $i$-th smallest Dirichlet eigenvalue of $H_{\xi}$ in $(-N,N)^d$. 
It is well known (see e.g.\ \cite{KM82a}) that the above limit exists in the sense of 
vague convergence and that its Laplace transform can be expressed as
\begin{equation*}
 \int_0^{\infty} e^{-t\lambda} \, \ell(d\lambda) 
 = (2 \pi t)^{-\frac{d}{2}}\int_{[0,\,1)^d} \E_{\theta} \otimes 
 E_x \left[ \exp\set{-\int_0^t V(B_s,\xi) \,ds} \Big| \,B_t=x \right]\,dx
\end{equation*}
using Brownian bridge measure. As one expects from this expression, 
it is not difficult to see that the right-hand side admits essentially the same upper and lower bounds 
as $S_t$ (see e.g.\ the discussion in \cite{Szn90a}): 
\begin{equation*}
 \begin{split}
  \log \int_0^{\infty} e^{-t\lambda} \, \ell(d\lambda) \asymp 
  \left\{
  \begin{array}{lr}
   -{t^{\frac{2+\theta}{4+\theta}}( \log t )^{-\frac{\theta}{4+\theta}}} &(d=2), \\[8pt]
   -{t^{\frac{d^2+2\theta}{d^2+2d+2\theta}}}  &(d \ge 3), 
  \end{array}\right. 
 \end{split}
\end{equation*}
as $t \to \infty$. 
From these asymptotics and the exponential Tauberian theorem due to Kasahara \cite{Kas78}, 
we find the following asymptotics for $\ell([0,\lambda])$. 
\begin{cor}\label{cor1} For any $\theta > 0$, we have 
 \begin{equation*}
  \begin{split}
   \log \ell([0,\lambda]) \asymp 
   \left\{
   \begin{array}{lr}
    -\lambda^{-1-\frac{\theta}{2}}\left( \log \frac{1}{\lambda} \right)^
    {-\frac{\theta}{2}} &(d=2),\\[8pt]
    -\lambda^{-\frac{d}{2}-\frac{\theta}{d}}  &(d \ge 3),
   \end{array}\right.
  \end{split}
 \end{equation*}
as $\lambda \to 0$.
\end{cor}
\noindent
This result says that the density of states is exponentially thin around the bottom of the spectrum, 
which is called ``the Lifshitz tail effect'' (cf.~\cite{Lif65}). 
Moreover, we find similar phenomena as in Remark \ref{disorder} 
for the power of $\lambda$. 
Namely, it approaches to ${d}/{2}$, the same power as in the Poissonian traps (cf.~\cite{Nak77}), 
in the limit $\theta \to 0$ 
and to $\infty$ in the limit $\theta \to \infty$, which corresponds to the periodic traps 
where the density of states vanishes near the origin. 
%%%%%%%%%%%%%%%%%%%%%%%%%%%%%%%%%%%%%%%%%%%%%%%%%%section3.2%%%%%%%%%%%%%%%%%%%%%%%%%%%%%%
\subsection{Intermittency}\label{Intermittency}
We consider the initial value problem
\begin{equation*}
 \frac{\partial}{\partial t} u(t, x) = H_{\xi} u(t, x) \quad \textrm{with} \quad u(0, \, \cdot \,) \equiv 1, \label{PAM}
\end{equation*}
which is called the ``parabolic Anderson problem''. 
The bounded solution $u_{\xi}$ of this problem is known to be unique and admits Feynman-Kac representation 
(see e.g.\ Chapter 1 of \cite{Szn98}). 
Therefore, we can identify $S_t$ with $\E_{\theta}[u_{\xi}(t, 0)]$. 
We analogously write the $p$-th moment by $S_{t,\,p}=\E_{\theta}[u_{\xi}(t, 0)^p]$. 
Then, the solution $u_{\xi}$ is said to be ``intermittent'' if 
\begin{equation*}
 \frac{S_{t,\,q}^{1/q}}{S_{t,\,p}^{1/p}} \xrightarrow{t \to \infty} \infty \quad\textrm{when}\quad p < q. 
\end{equation*}
Intermittency is usually regarded as an evidence of the strong inhomogeneity of the solution field. 
Indeed, if one considers a function consisting of a few high peaks, its $L^q$-norm tends to be much 
larger than its $L^p$-norm for $p < q$. For more on intermittency, see for instance \cite{GM90}. 

In our model, intermittency follows from Theorem 3.2(iii) of \cite{GM90}. 
Although it is stated in the discrete setting, 
the proof of this part of the theorem works in the continuous setting as well. 
Our aim in this section is to prove the following quantitative estimate for the moment asymptotics. 
In particular, it follows that small $\theta$ implies strong intermittency. 
\begin{cor}
Suppose that we have a modified potential \eqref{mptl}. Then for any $1 \le p < q$,  
\begin{equation}
 \limsup_{t \to \infty} \frac{\log S_{t,\,q}^{1/q}}{\log S_{t,\,p}^{1/p}} 
 \le \Bigl( \frac{p}{q} \Bigr)^{\frac{2}{d+\theta+2}}\frac{\limsup_{r \to \infty}M_r}{\liminf_{r \to \infty}M_r}. 
 \label{imt}
\end{equation}
\end{cor}
\begin{proof}
We prove this result only for $d \ge 3$ and the parameters $\e, L, h=1$. 
For the two-dimensional case, we have to care about the logarithmic correction 
but it is not difficult. 

The key to the proof is that we can prove 
\begin{equation*}
\begin{split}
 & t^{-\frac{d^2+2\theta}{d^2+2d+2\theta}}\log S_{t,\,p}\\ 
 \sim\, &-\inf_{(R_r,\, U_r) \in \smash[b]{\underline{\S}}_r}\Bigg\{p \lambda_1^r(U_r, R_r) 
  + \delta_c(r)^{-\theta}\int_{ R_r \setminus U_r } {\rm d}(x, \partial (R_r \setminus U_r))^{\theta} dx\Bigg\}
\end{split}
\end{equation*}
by exactly the same argument as for Theorem \ref{thm3}. 
Then, using the spatial scaling by the factor $p'=p^{d/(d^2+2d+2\theta)}$, we find that the right-hand side equals 
\begin{equation*}
 -p^{\frac{d^2+2\theta}{d^2+2d+2\theta}}\inf_{(R'_{r},\, U'_{r}) \in \smash[b]{\underline{\S}}_{rp'}}
 \Bigg\{ \lambda_1^{rp'}(U'_r, R'_r) 
 + \delta(rp')^{-\theta}\int_{ R'_r \setminus U'_r } {\rm d}(x, \partial (R'_r \setminus U'_r))^{\theta} dx\Bigg\}. \label{p-th}
\end{equation*}
Since the infimum in this expression is $M_{rp'}$, we obtain 
\begin{equation*}
 \limsup_{t \to \infty} \frac{\log S_{t,\,q}^{1/q}}{\log S_{t,\,p}^{1/p}} 
 \sim \Bigl( \frac{p}{q} \Bigr)^{\frac{2}{d+\theta+2}}\frac{M_{rq'}}{M_{rp'}} \quad {\rm as} \quad r \to \infty.
\end{equation*}
and \eqref{imt} follows. 
\end{proof}

\section*{Acknowledgements}
The author would like to thank professor Fr{\'e}d{\'e}ric Klopp for drawing his attention to \cite{BLS07}. 
The problem discussed in Appendix \ref{A} is suggested by professors Tokuzo Shiga, Nariyuki Minami, 
and Hiroshi Sugita. He appreciates them for the interesting problem. 
He is also grateful to the referees for careful reading of the article 
and constructive comments which substantially improved the exposition. 

% \appendix

% \section{}
% \label{}
\appendix
%%%%%%%%%%%%%%%%%%%%%%%%%%%%%%%%%%%%%%%%%%%%%%%%%%Appendix 1%%%%%%%%%%%%%%%%%%%%%%%%%%%%%%%%%%%%%%%%%%%
\section{}\label{A}
We discuss here weak convergences of the perturbed lattice as point processes. 
When we discuss weak convergence, we regard $(\P_{\theta})_{\theta>0}$ as probability measures 
on $\Xi$ equipped with the vague topology. 
Let $\P_{\infty}$ denote the perturbed lattice with the perturbation variables distributed uniformly on $B(0,1)$ 
and $\P_0$ the Poisson point process with unit intensity. 
\begin{thm}\label{thm8}
 $\P_{\theta}$ converges weakly to $\P_{\ast}$ $(\ast = 0 \textrm{ or }\infty)$ as $\theta \to \ast$. 
\end{thm}
\begin{rem}
It will be clear from the proof that we can make the perturbed lattice converge to 
the perfect lattice as $\theta \to \infty $ by changing the distribution 
of $\xi_q$ as
\begin{equation*}
 \P_{\theta}(\xi_q \in dx) = N'(d,\theta)\exp\{-(1+|x|)^{\theta}\}\,dx. 
\end{equation*}
 As we mentioned after \eqref{tail}, such a change does not affect the main results. 
 Therefore, we see again that our model interpolates a perfect crystal and a completely disordered medium 
 (cf.~Remark \ref{disorder}). \qed
\end{rem}
To prove Theorem \ref{thm8}, we use the following result concerning the convergence of point processes 
(see Theorem 4.7 of \cite{Kal86}).
\begin{lem}\label{lem3}
 Let $(\P_{\theta})_{\theta \in [0,\,\infty]}$ be a family of probability measures on $\Xi$. 
 Suppose that the following two conditions hold for any bounded Borel set $B \subset \R^d:$
 \begin{equation*}
  \begin{split}
   &{\rm (i)}\,\lim_{\theta \to \ast}\P_{\theta}(\xi(B)=0) = \P_{\ast}(\xi(B)=0), \\
   &{\rm (ii)}\,\limsup_{\theta \to \ast}\E_{\theta}[\xi(B)] \le \E_{\ast}[\xi(B)]. 
  \end{split}
 \end{equation*} 
 Then $\P_{\theta}$ converges weakly to $\P_{\ast}$ $(\ast = 0 \textrm{ or }\infty)$ as $\theta \to \ast$. 
\end{lem}

\noindent\textit{Proof of Theorem \ref{thm8}.} 
We consider the limit $\theta \to \infty$ first. 
In this case, the law of each $\xi_q$ converges to the uniform distribution on $B(0,1)$. 
Moreover, we have 
\begin{equation*}
 \P_{\theta}(q+\xi_q \in B) \le |B| N(d, \theta) \exp \bigl\{-{\rm d}(q,B)^{\theta}\bigr\} 
\end{equation*}
for any bounded $B \subset \R^d$. 
This implies that the law of $\xi(B)$ is essentially 
determined by finite $\xi_q$'s when $\theta$ is large. 
From these facts, it is easy to verify the conditions (i) and (ii) in Lemma \ref{lem3} and 
we have desired convergence. 

Next, we turn to more subtle case $\theta \to 0$. 
We first verify the condition (i), that is, $\lim_{\theta \to 0}\P_{\theta}(\xi(B)=0)=e^{-|B|}$. 
Let us take $M>0$ so large that $B \subset [-M,M]^d$. 
Then it follows that
\begin{equation*}
 \sup_{x \in B,\, q \not\in [-2M,2M]^d}\big||x-q|^{\theta}-|q|^{\theta}\big|
 \le 2 \theta M^{\theta}
\end{equation*}
for $\theta < 1$ from the mean value theorem. Therefore, for any $\e > 0$, we have 
\begin{equation}
 1-\e < \frac{\int_B \exp\{-|x-q|^{\theta}\}\, dx}{|B| \exp\{-|q|^{\theta}\}} < 1+\e \label{approx}
\end{equation} 
for all $q \notin [-2M,2M]^d$ when $\theta$ is sufficiently small. 
The right inequality in \eqref{approx} gives us an upper bound 
\begin{equation*}
 \begin{split}
  \P_{\theta}(\xi(B)=0) = &\, \prod_{q \in \Z^d}\biggl( 1 - N(d, \theta) \int_B \exp\{-|x-q|^{\theta}\}\, dx \biggr)\\
  \le &\, \prod_{q \notin [-2M,\,2M]^d}\biggl( 1 - (1-\e) N(d, \theta) |B| \exp\{-|q|^{\theta}\} \biggr). 
 \end{split}
\end{equation*}
Using $1-a \le e^{-a}$ in the right-hand side, we get 
\begin{equation}
 \begin{split}
  &\limsup_{\theta \to 0} \P_{\theta}(\xi(B)=0)\\ 
  \le &\, \limsup_{\theta \to 0} \exp\Biggl\{ -(1-\e) N(d, \theta) |B| \sum_{q \notin [-2M,\,2M]^d} \exp\{-|q|^{\theta}\} \Biggr\}\\ 
  = &\, \exp\{ -(1-\e) |B| \}. \label{u}
 \end{split}
\end{equation}
In the third line, we have used the fact $N(d, \theta) \sum_{q \notin [-2M,\,2M]^d} \exp\{-|q|^{\theta}\} \to 1 
\; (\theta \to 0)$, which can be shown by the same way as \eqref{approx}. 
For the lower bound, we use the left inequality in \eqref{approx} as follows: 
\begin{equation*}
 \begin{split}
  \P_{\theta}(\xi(B)=0) = &\, \prod_{q \in \Z^d}\biggl( 1 - N(d, \theta) \int_B \exp\{-|x-q|^{\theta}\}\, dx \biggr)\\
  \ge &\, \prod_{q \in [-2M,\,2M]} \biggl( 1 - N(d, \theta) \int_B \exp\{-|x-q|^{\theta}\}\, dx \biggr)\\
  &\, \times\prod_{q \notin [-2M,\,2M]^d}\biggl( 1 - (1+\e) N(d, \theta) |B| \exp\{-|q|^{\theta}\} \biggr). 
 \end{split}
\end{equation*}
Since $N(d,\theta) \to 0 \; (\theta \to 0)$, the first factor in the right-hand side 
goes to 1 and also 
\begin{equation*}
 \sup_{q \in \Z^d}(1+\e)N(d,\theta)|B|\exp\{-|q|^{\theta}\} \to 0 \quad {\rm as} \quad \theta \to 0. 
\end{equation*}
Thus, we can use $1-a \ge e^{-(1+\e)a}$, which is valid only for small $a>0$, in the second factor 
and get
\begin{equation}
 \begin{split}
  &\liminf_{\theta \to 0} \P_{\theta}(\xi(B)=0)\\ 
  \ge &\, \liminf_{\theta \to 0} \exp\Biggl\{ -(1+\e)^2 N(d, \theta) |B| \sum_{q \notin [-2M,\,2M]^d} \exp\{-|q|^{\theta}\} \Biggr\}\\ 
  = &\, \exp\{ -(1+\e)^2 |B| \}. \label{l}
 \end{split}
\end{equation}
Now that we have \eqref{u} and \eqref{l} for an arbitrary $\e>0$, the condition (i) is verified.  

Next, we proceed to check the condition (ii), that is, $\limsup_{\theta \to 0}\E_{\theta}[\xi(B)] \le |B|$. 
Using the right inequality in \eqref{approx}, we find
\begin{equation*}
 \begin{split}
  \E_{\theta}[\xi(B)] = &\, \sum_{q \in \Z^d} \P_{\theta}(q+\xi_q \in B)\\
  =&\, N(d,\theta) \sum_{q \in \Z^d}\int_B \exp\{-|x-q|^{\theta}\}\, dx\\ 
  \le &\, (1-\e)^{-1} |B| N(d,\theta) \Biggl( \sum_{q \notin [-2M,\,2M]^d} \exp\{-|q|^{\theta}\} + (4M)^d\Biggr).  
 \end{split}
\end{equation*} 
Since the right-hand side of this inequality goes to $(1-\e)^{-1} |B|$ as $\theta \to 0$, we are done. \qed

%%%%%%%%%%%%%%%%%%%%%%%%%%%%%%%%%%%%%%%%%%%%%%%%%%Appendix 2%%%%%%%%%%%%%%%%%%%%%%%%%%%%%%%%
\section{}\label{B}
In this appendix, we prove the finiteness of 
\begin{equation*}
 \sup_{r \ge 1} M_r = \sup_{r \ge 1}\inf_{(R_r,\, U_r) \in \smash[b]{\underline{\S}}_r'}\bigg\{\lambda_1^r(U_r, R_r)
 + \delta_c(r)^{-\theta}\int_{ R_r \setminus U_r } {\rm d}(x, \partial (R_r \setminus U_r))^{\theta}dx \bigg\}. 
\end{equation*}
To this end, we have only to find a specific sequence $\{(R_r,U_r) \in \underline{\S}_r'\}_{r \ge1} $ 
for which the functional in the infimum is bounded above. 
We take $R_r=(-1,1)^d$ and $U_r$ as in \eqref{punched} with $\delta(r)=\delta_c(r)$, 
the critical interval. 
This critical regime is called the ``constant capacity regime'' (see Section 3.2.B in \cite{Szn98}). 
Note that $R_r$ is not a \underline{\smash[b]{density}} box, with a slight abuse of notation, 
and thus $(R_r,U_r) \in \underline{\S}_r'$. 
To be honest, we have to rearrange $U_r$ a little so that each cubes 
being centered on $r^{-1} \Z^d$ but we ignore this point for simplicity. 

Firstly, it is easy to see that
\begin{equation*}
\delta_c(r)^{-\theta}\int_{ R_r \setminus U_r } {\rm d}(x, \partial (R_r \setminus U_r))^{\theta}dx
\asymp 1 \quad {\rm as} \quad r \to \infty. 
\end{equation*}
Thus, it suffices to show that the principal Dirichlet eigenvalue 
$\lambda_1^r(\emptyset, R_r \setminus U_r) \ge \lambda_1^r(U_r, R_r)$ is bounded above. 
This assertion is shown in Theorem 22.1 of \cite{Sim05} in the case $d = 3$. 
Since the same proof directly applies to all $d \ge 3$, we restrict the discussion on $d=2$. 
Let $\psi$ be the $L^2$-normalized principal Dirichlet eigenfunction in $(-1,1)^d$ and 
\begin{equation*}
 \phi_r(x)=\prod_{q \in \Z^d} \left(\frac{\log |x-\delta_c(r)q|-\log (1/r)}{\log (\delta_c(r)/2)-\log (1/r)} \right)_{+}\wedge 1.
\end{equation*}
Then it follows from the direct calculation that for arbitrary small $\e > 0$, 
\begin{equation*}
\inf\biggl\{\phi_r(x); x \in (-1,1)^d \setminus \bigcup_{q \in \Z^d}C(\delta_c(r)q , \e \delta_c(r))\biggr\} \to 1 
\end{equation*}
as $r \to \infty$. Moreover, it is not difficult to show that both $\|(\nabla \psi)\phi_r\|_2$ and 
$\|\psi (\nabla \phi_r)\|_2$ are bounded. 
We combine these three estimates to bound the right-hand side of 
\begin{equation*}
 \lambda_1^r(\emptyset, R_r \setminus U_r)
 \le \frac{1}{\| \psi\phi_r\|_2^2} \int_{U_r}\frac{1}{2}|\nabla (\psi\phi_r) |^2(x) \, dx
\end{equation*}
and get the desired result.

\end{document}